\documentclass[11pt]{article}

\usepackage[latin1]{inputenc}
\usepackage{amsmath,amssymb,amsthm}
\usepackage{bbm}
\usepackage{enumerate}

\usepackage{picture} 
\usepackage{epic} 
\usepackage{graphicx}
\usepackage{wrapfig}

\usepackage{marginnote}

\usepackage[numbers]{natbib}
\usepackage[font=small,labelfont=bf]{caption} 

\usepackage{wrapfig}

\newtheorem{theorem}{Theorem}
\newtheorem{lemma}[theorem]{Lemma}
\newtheorem{definition}[theorem]{Definition}
\newtheorem{corollary}[theorem]{Corollary}
\newtheorem{proposition}[theorem]{Proposition}

\newcommand{\Pb}{\mathbb{P}}

\newcommand{\Z}{\mathbb{Z}}
\newcommand{\N}{\mathbb{N}}
\newcommand{\R}{\mathbb{R}}
\newcommand{\Bin}{\mathrm{Bin}}
\newcommand{\rc}{\mathrm{rc}}
\newcommand{\diam}{\mathrm{diam}}
\title{The hitting time of rainbow connection number two}
\date{September 13, 2012}

\renewcommand{\leq}{\leqslant} 
\renewcommand{\geq}{\geqslant}
\renewcommand{\le}{\leqslant} 
\renewcommand{\ge}{\geqslant}
\renewcommand{\epsilon}{\varepsilon}

\author{Annika Heckel and Oliver Riordan
\thanks{Mathematical Institute, University of Oxford,
24--29 St Giles', Oxford OX1 3LB, UK. E-mail: 
\texttt{$\{$heckel,riordan$\}$@maths.ox.ac.uk}
}
}

\begin{document}
\maketitle

\begin{abstract}
In a graph $G$ with a given edge colouring, a rainbow path is a path all of whose edges have distinct colours. The minimum number of colours required to colour the edges of $G$ so that every pair of vertices is joined by at least one rainbow path is called the rainbow connection number $\rc(G)$ of the graph $G$. For any graph $G$, $\rc(G) \geq \diam(G)$. We will show that for the Erd\H{o}s--R{\'e}nyi random graph $\mathcal{G}(n,p)$ close to the diameter $2$ threshold, with high probability if $\diam(G)=2$ then $\rc(G)=2$. In fact, further strengthening this result, we will show that in the random graph process, with high probability the hitting times of diameter $2$ and of rainbow connection number $2$ coincide.
\end{abstract}

\section{Introduction}

The rainbow connection number is a new concept for measuring the connectivity of a graph which was introduced by Chartrand, Johns, McKeon and Zhang in \cite{chartrand:rainbow}. In a graph $G$ with a given edge colouring, we call a path a \emph{rainbow path} if all its edges have distinct colours. We call the colouring a \emph{rainbow colouring} if every pair of vertices is joined by at least one rainbow path. The minimum number of colours required for such a colouring is called the \emph{rainbow connection number} (or \emph{rainbow connectivity}) $\rc(G)$ of the graph $G$. Rainbow colourings have received considerable attention since their introduction, being both of theoretical interest and highly applicable. A recent account of known results in this area is given in \cite{li:rainbowsurvey}.

A trivial lower bound for the rainbow connection number of a graph is its diameter, as pointed out in \cite{chartrand:rainbow}: In a rainbow colouring with $k$ colours, every pair of vertices is joined by a path of length at most $k$.

We will study rainbow connection numbers in the random graph setting. More specifically, for $n\in \N$ and $p \in [0,1]$, we consider the  Erd\H{o}s--R{\'e}nyi random graph model, denoted by $G \sim \mathcal{G}(n,p)$, which is a graph with $n$ vertices where each of the $\binom{n}{2}$ potential edges is present with probability $p$, independently. We say that an event $E=E(n)$ holds \emph{with high probability (whp)} if $\lim_{n \rightarrow \infty} \Pb(E(n)) = 1$. We call a sequence $p^*(n)$, $n \in \N$, a \emph{semisharp threshold} for a graph 
property $\mathcal{P}$ if there are constants $c,C>0$ such that if $p(n) \geq C p^*(n)$ for all $n$, then whp $\mathcal{G}(n,p(n)) \in \mathcal{P}$, and if $p(n) \leq c p^*(n)$ for all $n$, then whp $\mathcal{G}(n,p(n)) \notin \mathcal{P}$. This (non-standard) terminology reflects the fact that this notion is in between that of a truly sharp threshold, where these properties hold for any $C>1$ and any $c<1$, and of a (weak) threshold, where the conditions assume $p(n)/p^*(n) \rightarrow \infty$ and $p(n)/p^*(n) \rightarrow 0$.

Caro, Lev, Roditty, Tuza and Yuster \cite{caro:rainbow} showed that $\sqrt{\frac{\log n}{n}}$ is a semisharp threshold for the property $\rc(G) \leq 2$. This result was generalised by He and Liang \cite{he:rainbow} who showed that for any constant $d \in \N$, $\frac{(\log n)^{1/d}}{n^{1-1/d}}$ is a semisharp threshold for the property $\rc(G) \leq d$. Both of these results rely on random colourings. Since, as shown by Bollob\'as \cite{bollobas:diameter}, $\frac{(2\log n)^{1/d}}{n^{1-1/d}}$  is a sharp
threshold for the property $\diam(G)\le d$, a natural question is whether $\rc(G)\le d$ has the same \emph{sharp} threshold.

In a different direction, recently Frieze and Tsourakakis \cite{frieze:rainbow} showed that at the connectivity threshold $\frac{\log n + \omega(n)}{n}$ where $\omega(n) \rightarrow \infty$ and $\omega(n)=o(\log n)$, the rainbow connection number of a random graph is whp asymptotically $\max\{Z_1, \diam(G)\}$, where $Z_1$ denotes the number of degree $1$ vertices of the graph.

For $d=2$ we shall answer the question above in the strongest possible sense, showing that rainbow connection number $2$ occurs essentially at the same time as diameter $2$ in random graphs, and indeed even in the random graph process. To do this, we shall consider colourings constructed in two rounds, the first uniformly random, and the second `more intelligent'. We will first consider $\mathcal{G}(n,p)$ close to the threshold for diameter $2$.

\begin{theorem}
\label{theoremforrandomgraphs}
 Let $p=p(n)=\sqrt{\frac{2\log n+\omega(n)}{n}}$ where $\omega(n)=o(\log n)$ and let $G \sim \mathcal{G} (n,p)$. Then whp $\rc(G)=\diam(G)\in \{2,3\}$.
\end{theorem}

From \cite{bollobas:diameter} (see also Theorem 10.10 and Corollary 10.11 in \cite{bollobas:randomgraphs}), we immediately get the following corollaries.

\begin{corollary}
Let $p = \sqrt{\frac{2 \log n + c}{n}}$ where $c \in \R$ is a constant, and let $G \sim \mathcal{G} (n,p)$. Then $ \lim_{n \rightarrow \infty }\Pb(\rc (G) = 2) = e^{-e^{-c}/2}$ and $\lim_{n \rightarrow \infty }\Pb (\rc (G) = 3) = 1-e^{-e^{-c}/2}$.
\end{corollary}
\begin{corollary}
Let $p = \sqrt{\frac{2 \log n + \omega(n)}{n}}$ where $\omega(n)\rightarrow \infty$ such that $(1-p)n^2 \rightarrow \infty$, and let $G \sim \mathcal{G} (n,p)$. Then $\rc(G) = 2$ whp.
\end{corollary}

We will in fact prove something even stronger than Theorem~\ref{theoremforrandomgraphs}. Consider the \emph{random graph process} $(G_t)_{t=0}^N$, $N= \binom{n}{2}$, which starts with the empty graph on $n$ vertices at time $t=0$ and where at each step one edge is added, chosen uniformly at random from those not already present in the graph, until at time $N$ we have a complete graph. A graph property is called \emph{monotone increasing} if it is preserved under the addition of further edges to a graph. For a monotone increasing graph property $\mathcal{P}$, let $\tau_{\mathcal{P}}$ be the \emph{hitting time} of $\mathcal{P}$, i.e.\, the smallest $t$ such that $G_t$ has property $\mathcal{P}$.

Consider the graph properties $\mathcal{D}$ and $\mathcal{R}$ given by
\begin{eqnarray*}\mathcal{D} &=& \{ G: \diam(G) \le 2 \} \\
 \mathcal{R} &=&  \{ G: \rc(G) \le 2 \}.
\end{eqnarray*}
Then $\mathcal{D}$ and $\mathcal{R}$ are monotone increasing. Since $\mathcal{D}$ is necessary for $\mathcal{R}$, we always have $\tau_\mathcal{D} \leq \tau_\mathcal{R}$; we will prove that whp $\mathcal{D}$ and $\mathcal{R}$ occur at the same time.

\begin{theorem}
\label{theoremforhittingtimes}
In the random graph process $(G_t)_{t=0}^N$, with high probability $\tau_\mathcal{D} = \tau_\mathcal{R}$.
\end{theorem}

For the proofs of these theorems, we will need a number of definitions. In a graph $G$ with a given edge $2$-colouring, we call a pair of non-adjacent vertices \emph{dangerous} if they are joined by at most $d=66$ rainbow paths of length $2$. Moreover, we call a pair of non-adjacent vertices \emph{sparsely connected} if they are joined by at most $d=66$ paths of length $2$ (rainbow or otherwise) and \emph{richly connected} otherwise.

\begin{definition}
We say that a graph has \emph{property $\mathcal{M}$} if it has a spanning subgraph which has an edge 2-colouring such that
\begin{enumerate}[(i)]
 \item Every vertex is in at most $3$ dangerous pairs.
 \item Every vertex is joined by edges to both vertices of at most $15$ dangerous pairs.
 \item Every vertex is in at most one sparsely connected pair.
\end{enumerate}
\end{definition}
Note that $\mathcal{M}$ is a monotone increasing graph property because it is defined by the existence of a spanning subgraph with some property. The property of having a colouring satisfying conditions (i)--(iii) is not itself monotone increasing, since condition (ii) does not necessarily stay true if we add more edges.

The following two propositions will form the main part of our proof.

\begin{proposition}
\label{propertymbutnotd}
If $p=\sqrt{\frac{1.99 \log n}{n}}$, then whp the graph $G \sim \mathcal{G} (n,p)$ has property $\mathcal{M}$.
\end{proposition}

\begin{proposition}
\label{manddimpliesr}
 If a graph has properties $\mathcal{M}$ and $\mathcal{D}$, it also has property $\mathcal{R}$.
\end{proposition}

\section{Proofs}

Before turning to the proofs of Propositions~\ref{propertymbutnotd} and \ref{manddimpliesr}, we will show how they can be used to prove Theorems~\ref{theoremforrandomgraphs} and \ref{theoremforhittingtimes}.

\subsection{Proofs of Theorems~\ref{theoremforrandomgraphs} and \ref{theoremforhittingtimes}}

\begin{proof}[Proof of Theorem~\ref{theoremforrandomgraphs}]
Let $p=p(n)=\sqrt{\frac{2\log n+\omega(n)}{n}}$ where $\omega(n)=o(\log n)$, and let $G \sim \mathcal{G}(n,p)$. Since $p$ is well above the threshold $\frac{(\log n)^{1/3}}{n^{2/3}}$ for the property $\rc(G)\le 3$ established by He and Liang \cite{he:rainbow}, we certainly have $\rc(G)\le 3$ whp. In fact, for this $p$, it is easy to check that a random 3-colouring is rainbow whp.
Since ${p}^{\binom{n}{2}}=o(1)$, whp $G$ is not complete, so whp $\diam(G) \ge 2$. Since $\diam (G) \le \rc(G)$, it remains only to show that whp $\diam(G)=2$ implies $\rc(G)=2$.

For $n$ large enough, $p \geq \sqrt{\frac{1.99 \log n}{n}}$. Since $\mathcal{M}$ is monotone increasing, it follows from Proposition~\ref{propertymbutnotd} that whp $G$ has property $\mathcal{M}$. By Proposition~\ref{manddimpliesr}, if $\diam(G)=2$, i.e., $G$ has property $\mathcal{D}$, then $G$ also has property $\mathcal{R}$, so its rainbow connection number is at most $2$.
\end{proof}

For Theorem~\ref{theoremforhittingtimes}, we need to construct the random graph process so that we can couple it with $\mathcal{G}(n,p)$, $p \in [0,1]$.

\begin{proof}[Proof of Theorem~\ref{theoremforhittingtimes}]

Take a set $V$ of vertices where $|V|=n$, and assign to each potential edge $e$ a random variable $X_e$ which is distributed uniformly on $[0,1]$, independently. Order the potential edges in ascending order of the corresponding random variables $X_e$. Almost surely, no two of the $X_e$ take the same value, and any order of the $X_e$ is equally likely. Therefore, we can add the edges to the graph one-by-one in the ascending order of the corresponding $X_e$, yielding a random graph process $(G_t)_{t=0}^N$, $N= \binom{n}{2}$, with the required distribution.

Let $p=\sqrt{\frac{1.99 \log n}{n}}$ and let $G=(V,E)$ where $e\in E$ iff $X_e\leq p$. Then since the random variables $X_e$ are i.i.d.\ and distributed uniformly on $[0,1]$, $G \sim \mathcal{G}(n,p)$.

By Proposition~\ref{propertymbutnotd}, whp $G$ has property $\mathcal{M}$. Since, as shown by Bollob{\'a}s \cite{bollobas:diameter} (see Corollary 10.11 in \cite{bollobas:randomgraphs}), $\sqrt{\frac{2\log n}{n}}$ is a sharp threshold for the property $\mathcal{D}$, whp $G$ does not have property $\mathcal{D}$.


Since in the random graph process we added the edges in ascending order of their corresponding random variables, there is a (random) time $0\leq t \leq N$ such that $G=G_t$. Therefore, there is whp a time $t$ such that $G_t$ has property $\mathcal{M}$ but not property $\mathcal{D}$, so $\tau_\mathcal{M} < \tau_\mathcal{D}$ whp. From Proposition~\ref{manddimpliesr}, we get $\tau_\mathcal{R} \leq \max \lbrace \tau_\mathcal{D}, \tau_\mathcal{M} \rbrace = \tau_\mathcal{D}$ whp, and together with the trivial observation $\tau_\mathcal{D} \leq \tau_\mathcal{R}$, this implies the result. 
\end{proof}

\subsection{Bounds for binomial distributions}

For the proof of Proposition~\ref{propertymbutnotd}, we shall need some preliminary lemmas concerning bounds for binomial distributions. Recall the well-known Chernoff bounds (\cite{chernoffbound}, see also \cite[p.26]{janson:randomgraphs}).

\begin{lemma}
\label{chernoffbound}
Let $X$ be a random variable, distributed binomially with parameters $n \in \N$ and $p \in (0,1)$, and let $0<x<1$. 
\begin{enumerate}[(i)]
 \item If $x \geq p$, then $\Pb(X \geq nx) \leq \left[ \left( \frac{p}{x}  \right)^x \left( \frac{1-p}{1-x}\right)^{1-x} \right]^n$.
 \item If $x \leq p$, then $\Pb(X \leq nx) \leq \left[ \left( \frac{p}{x}  \right)^x \left( \frac{1-p}{1-x}\right)^{1-x} \right]^n$.
\end{enumerate}
\end{lemma}

A more convenient bound is given by the following corollary (see \cite[p.27]{janson:randomgraphs}).
\begin{corollary}
\label{chernoffcorollary}
Let $X$ be a random variable, distributed binomially with parameters $n \in \N$ and $p \in [0,1]$. If $0 <\epsilon \leq \frac{3}{2}$, then
\[
\Pb\left( \left| X - np \right| \geq \epsilon np \right) \leq 2 e^{-\epsilon^2 np /3}.
\]
\end{corollary}
We will also need another consequence of the Chernoff bounds.
\begin{corollary}
\label{binomialbound}
Let $(n_i)_{i \in \Z}$ be a sequence of integers such that $ n_i \rightarrow \infty$ as $i \rightarrow \infty$, and let $(p_i)_{i \in \N}$ be a sequence of probabilities. Let $X_i \sim \Bin (n_i, p_i)$, and let $k \in \N$ be constant. Suppose that $\mu_i := n_i p_i \rightarrow \infty$ as $i \rightarrow \infty$. Then
\[
\Pb(X_i \leq k) = O(\mu_i^k e^{-\mu_i}).
\]
\end{corollary}

\begin{proof}
Applying Lemma~\ref{chernoffbound} to $X_i$ with $x_i = \frac{k}{n_i}$ gives
\begin{eqnarray*}
\Pb(X_i \leq k) &=& \Pb(X_i \leq n_i x_i) \leq  \left( \frac{\mu_i}{k}  \right)^k \left( \frac{1-p_i}{1-\frac{k}{n_i}}\right)^{n_i-k} = O\left(\mu_i^k \cdot \frac{e^{-\mu_i + p_i k}}{e^{-k}}\right) \\
&=&  O(e^{-\mu_i} \mu_i^k),
\end{eqnarray*}
using the fact that $1-y \leq e^{-y}$ and that $\lim_{n \rightarrow \infty} (1-\frac{y}{n})^n = e^{-y}$ for every $y \in \R$.
\end{proof}

\subsection{Proof of Proposition~\ref{propertymbutnotd}}

\subsubsection{Overview}

We will generate the graph and an edge $2$-colouring together in two steps. First consider the random graph $G_1 \sim \mathcal{G}(n,p_1)$ where $p_1 = \sqrt {\frac{(1+\epsilon)\log n}{n}}$ and $\epsilon = 0.01$. We will colour the edges of this graph randomly.

Next, we will add more edges to generate $G_2 \sim \mathcal{G}(n,p)$ where $p= \sqrt{\frac{1.99 \log n}{n}}$. Each edge which is not already present will be added independently with probability $p_2$, where $1-p = (1-p_1)(1-p_2)$.  We will colour these new edges so they add a rainbow $2$-path to a dangerous pair whenever possible. We will show in Lemma~\ref{atmost3}, Corollary~\ref{atmost15corollary} and Lemma~\ref{atmost1} that whp this gives an edge colouring which fulfills conditions (i)--(iii) of property $\mathcal{M}$ (with $G_2$ itself as the spanning subgraph).

\subsubsection{First step: A random colouring}

Let $G_1 \sim \mathcal{G}(n,p_1)$ where $p_1 = \sqrt {\frac{(1+\epsilon)\log n}{n}}$ and $\epsilon = 0.01$. Colour the edges of $G_1$ using two colours independently and uniformly at random.

For a given pair $\{ v,w \}$ of vertices and another vertex $z \notin \{ v,w \}$, the probability that there is a rainbow path from $v$ to $w$ via $z$ is $\frac{1}{2} p_1^2 = \frac{(1+\epsilon) \log n}{2n} $, and this is independent for different $z$. Therefore, the number of rainbow paths of length $2$ joining $v$ and $w$ is distributed binomially with parameters $n-2$ and $\frac{(1+\epsilon) \log n}{2n} $ and so has mean $\mu= \frac{(1+\epsilon)}{2}\log n +o(1)$. By Corollary~\ref{binomialbound},
\[
 \Pb(\{ v,w\} \text{ is dangerous in $G_1$}) = O(n^{-\frac{1}{2}(1+\epsilon)} (\log n)^d).
\]
We will now gather some information about the structure of the random graph and of the dangerous pairs in $G_1$. We denote the \textit{neighbourhood} of a vertex $v$ in a graph $G$ by $\Gamma(v)$.

\begin{lemma}
\label{numberofneigbours}
 With probability $1 - o(n^{-2})$, for every vertex $v$ in $G_1$, 
\[\sqrt {\left(1+\frac{\epsilon}{2}\right) n \log n} \leq |\Gamma(v)| \leq \sqrt {(1+2 \epsilon) n \log n}.\]
\end{lemma}
\begin{proof}
For a given vertex $v$, the number of neighbours of $v$ is binomially distributed with parameters $n-1$ and $p_1$ and has mean $(n-1)p_1 = \sqrt{\frac{(n-1)^2}{n}(1+\epsilon) \log n} \sim \sqrt{(1+\epsilon) n \log n }$. By Corollary~\ref{chernoffcorollary}, the probability that $v$ has more than $\sqrt {(1+2 \epsilon) n \log n}$ or fewer than $\sqrt {(1+\frac{\epsilon}{2}) n \log n}$ neighbours is $o(n^{-3})$. Taking the union bound over all vertices gives the result.
\end{proof}

\begin{lemma}
\label{numberofdangerous}
The probability that a given pair $\{v,w\}$ is dangerous in $G_1$ is $O(n^{-\frac{1}{2}(1 + \frac{\epsilon}{2})})$. Moreover, with probability $1-o(n^{-2})$, every vertex in $G_1$ is in at most $n^{\frac{1}{2}(1 - \frac{\epsilon}{4})}$ dangerous pairs.
\end{lemma}

\begin{proof}
Fix a vertex $v$ and explore the graph in the following way. Test all edges incident with $v$ and their colours. With probability $1-o(n^{-3})$, $|\Gamma(v)| \geq \sqrt {(1+\frac{\epsilon}{2}) n \log n}$ as in the proof of Lemma~\ref{numberofneigbours}. Assume this is the case.

Now let $w$ be a vertex with $w \notin \Gamma(v) \cup  \{v\}$. The number of edges between $w$ and $\Gamma(v)$ which have the correct colour for a rainbow $2$-path between $w$ and $v$ is distributed binomially with parameters $|\Gamma(v)|$ and $\frac{1}{2}p_1$, with mean at least $\frac{1}{2}\sqrt {(1+\epsilon)(1+\frac{\epsilon}{2})} \log n$. So the probability that $w$ has at most $d$ edges of the appropriate colour for a rainbow path to $\Gamma(v)$ is $O(n^{-\frac{1}{2}\sqrt{(1+\epsilon)(1+\frac{\epsilon}{2})}} (\log n)^{d})$ by Corollary~\ref{binomialbound}.

Therefore, $\lbrace v,w \rbrace$ is dangerous with probability $O(n^{-\frac{1}{2}(1 + \frac{\epsilon}{2})})$, and this happens independently for different $w \notin \Gamma(v) \cup \{ v \}$. So the number of dangerous pairs that $v$ is in is dominated by a binomial random variable with parameters $n$ and $O(n^{-\frac{1}{2}(1 + \frac{\epsilon}{2})})$, which has mean $O(n^{\frac{1}{2}(1 - \frac{\epsilon}{2})})$. By Corollary~\ref{chernoffcorollary}, with probability $1- o(n^{-3})$, $v$ is in at most $n^{\frac{1}{2}(1 - \frac{\epsilon}{4})}$ dangerous pairs. Taking the union bound over all $v$ gives the result.
\end{proof}

We call a pair of non-adjacent vertices $\lbrace x,y\rbrace$ in $G_1$ a \emph{fix} for a pair $\lbrace v,w \rbrace$ if adding an edge $e=xy$ of a certain colour would add a rainbow path of length $2$ between $v$ and $w$. We call a fix $\lbrace x,y\rbrace$ for a pair $\lbrace v,w \rbrace$ an \emph{exclusive fix} if there is no other dangerous pair (other than possibly $\lbrace v,w \rbrace$ if $\lbrace v,w \rbrace$ is dangerous) that $\lbrace x,y\rbrace$ is a fix for.

We expect to have about $2 n p_1$ fixes for every pair $\{v,w\}$ (of the form $\{x,w\}$ where $x \in \Gamma(v)$ or $\{v,y\}$ where $y \in \Gamma(w)$). We will now show that in fact most of these fixes are exclusive.

\begin{lemma}
\label{manyexclusivefixes}
Whp, every pair of vertices in $G_1$ is either adjacent or has at least $2 \sqrt {(1 + \frac{\epsilon}{4} ) n \log n }$ exclusive fixes.
\end{lemma}

\begin{proof}

Consider a pair of vertices $\lbrace v,w \rbrace$. Take $v$ out of the graph $G_1$ and just look at the remaining graph $G_1'$. Then by Lemma~\ref{numberofneigbours} and (a slight variant of) Lemma~\ref{numberofdangerous}, with probability $1-o(n^{-2})$, every vertex in $G_1'$ has at most $\sqrt {(1+2 \epsilon) n \log n}$ neighbours and is in at most $n^{\frac{1}{2}(1 - \frac{\epsilon}{4})}$ dangerous pairs (dangerous within $G_1'$).

In particular, if $E'_1$ denotes the set of vertices $x$ such that $x$ is in a dangerous pair (within $G_1'$) with a neighbour of $w$, and $E'_2$ denotes the set of vertices $x$ such that $x$ is a neighbour of a vertex that is in a dangerous pair (within $G_1'$) with $w$, then with probability $1-o(n^{-2})$, $|E'_1| \leq n^{1-\frac{\epsilon}{16}}$ and $|E'_2|  \leq n^{1-\frac{\epsilon}{16}} $.

In the whole graph $G_1$, let $E_1$ denote the set of all $x \in V \setminus \lbrace v, w \rbrace$ such that there is a neighbour $k \neq v$ of $w$ such that $\lbrace x, k \rbrace$ is dangerous (in $G_1$). Let $E_2$  denote the set of all $x \in V \setminus \lbrace v, w \rbrace$ which have a neighbour $l \neq v$ such that $\lbrace l,w \rbrace$ is dangerous (in $G_1$). Any pair $\{s,t\}\subset V\setminus \{v\}$ which is dangerous in $G_1$ is also dangerous in $G_1'$. Therefore, $E_1 \subset E'_1$ and $E_2 \subset E'_2$.

A pair $\lbrace x, w \rbrace$ where $x \in \Gamma(v)$ can only fail to be an exclusive fix for $\lbrace v, w \rbrace$ in one of the following three ways. Either $x$ and $w$ are adjacent, or there is a $k \in \Gamma(w)$ such that $\lbrace x , k \rbrace$ is dangerous, or there is an $l \in \Gamma(x) \setminus \lbrace v \rbrace$ such that $\{ l, w \}$ is dangerous (see Figure~\ref{nonexclusivefix}). If $v$ and $w$ are not adjacent, this can only happen if $x \in E_1 \cup E_2 \cup \Gamma'(w) \subset E'_1 \cup E'_2 \cup \Gamma'(w)$, where $\Gamma'(w)$ denotes the neighbourhood of $w$ in $G_1'$.

%

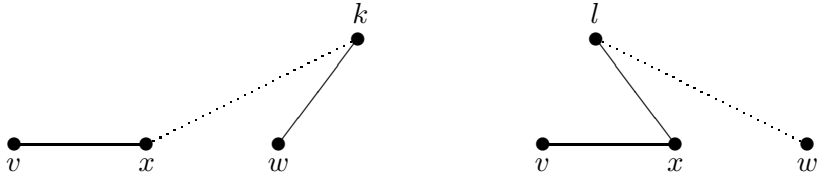
\begin{figure}[htb]
\begin{center}
\begin{picture}(300,60)
\unitlength = 1pt
\put(0,0){\circle*{5}}
\put(50,0){\circle*{5}}
\put(100,0){\circle*{5}}
\put(0,0){\line(1,0){50}}
\put(100,0){\line(3,4){30}}
\put(130,40){\circle*{5}}
\dottedline{3}(50,0)(130,40) 
\put(-3,-10){$v$}
\put(96,-10){$w$}
\put(47,-10){$x$}
\put(128,46){$k$}
\put(200,0){\circle*{5}}
\put(250,0){\circle*{5}}
\put(300,0){\circle*{5}}
\put(220,40){\circle*{5}}
\put(200,0){\line(1,0){50}}
\put(250,0){\line(-3,4){30}}
\dottedline{3}(220,40)(300,0)
\put(197,-10){$v$}
\put(296,-10){$w$}
\put(247,-10){$x$}
\put(218,46){$l$}
\end{picture}
\end{center}
\caption{\label{nonexclusivefix}Two ways in which $\{x,w\}$ can be a fix for a dangerous pair other than $\{v,w\}$. The dotted lines show dangerous pairs.}
\end{figure}

Condition on $G_1'$. With probability $1-o(n^{-2})$, $|E'_1 \cup E'_2 \cup \Gamma'(w)| \leq 3 n^{1-\frac{\epsilon}{16}}$. If this is the case, there are at least $n-2-3 n^{1-\frac{\epsilon}{16}}$ potential neighbours $x$ of $v$ such that $\lbrace x,  w \rbrace$ would be an exclusive fix for $\{v,w\}$; and each is actually adjacent to $v$ with probability $p_1$ independently of each other and of $G_1'$.

Therefore, if $v$ and $w$ are not adjacent,  the number of $x \in \Gamma(v)$ such that $\lbrace x, w \rbrace$ is an exclusive fix for $\lbrace v, w \rbrace$ is bounded from below by a binomial random variable with parameters $n-2-3 n^{1-\frac{\epsilon}{16}}$ and $p_1 = \sqrt {\frac{(1+\epsilon)\log n}{n}}$, which has mean greater than  $\sqrt {(1+\frac{\epsilon}{2})n \log n}$ if $n$ is large enough. By Corollary~\ref{chernoffcorollary}, it follows that with probability $1-o(n^{-2})$, there are at least $\sqrt {(1 + \frac{\epsilon}{4} ) n \log n }$ exclusive fixes of the form $\lbrace x, w \rbrace$ where $x \in \Gamma(v)$. Analogously, with probability $1-o(n^{-2})$, there are at least $\sqrt {(1 + \frac{\epsilon}{4} ) n \log n }$ exclusive fixes of the form $\lbrace v, y \rbrace$ where $y \in \Gamma(w)$, so overall with probability $1-o(n^{-2})$, there are at least $2 \sqrt {(1 + \frac{\epsilon}{4} ) n \log n }$ exclusive fixes for $\{v,w\}$.
\end{proof}

\subsubsection{Second step: More edges with a more intelligent colouring}

Now we are ready to introduce some additional edges which will be coloured more intelligently. Each edge which is not already present in the graph is now added independently with probability $p_2$, where $p_2$ is chosen so that $1-p = (1-p_1)(1-p_2)$. This ensures that after the second step, the probability that a particular edge is present is exactly $p = \sqrt{\frac{1.99 \log n}{n}}$.

Note that $p_2 = p - p_1 +p_1p_2 \geq p - p_1 = \frac{\sqrt{1.99 \log n} - \sqrt{(1 + \epsilon)\log n}}{\sqrt{n}} \geq 0.4 \cdot \sqrt{\frac{ \log n}{n}}$ (recall that $\epsilon = 0.01$).

Whenever a new edge is a fix for a dangerous pair, we give it the appropriate colour so that it adds a rainbow path of length $2$ joining the dangerous pair. If there are several such dangerous pairs, we pick any colour (or the colour that suits the most, etc., it does not matter).

By Lemma~\ref{manyexclusivefixes}, whp in $G_1$ there are at least $2 \sqrt {(1 + \frac{\epsilon}{4} ) n \log n }$ exclusive fixes for every dangerous pair. Assume this from now on. These exclusive fixes will always get the correct colour for this pair if they are added. For a dangerous pair $\{ v,w\}$ in $G_1$, let $N_{\{ v,w\}}$ be the number of exclusive fixes of $\{ v,w\}$ added in the second step. By definition, the sets of exclusive fixes are disjoint for different dangerous pairs. Therefore, conditional on $G_1$, the random variables $N_{\{ v,w\}}$ are independent.

For a fixed dangerous pair $\{v,w\}$ in $G_1$, $N_{\{ v,w\}}$ is bounded from below by a binomial random variable with parameters $2 \sqrt {(1 + \frac{\epsilon}{4} ) n \log n }$ and $p_2$, which has mean at least $0.8 \sqrt{1 + \frac{\epsilon}{4}} \log n$. Therefore, by Corollary~\ref{binomialbound},
\begin{equation}
\label{atmostdin2ndstep}
 \Pb (N_{\{ v,w\}} \leq d) = O(n^{-0.8\cdot \sqrt{1 + \frac{\epsilon}{4} }} (\log n)^{d}) = O(n^{-0.8}).
\end{equation}

\begin{lemma}
\label{atmost3}
 In $G_2$ whp no vertex is in more than three dangerous pairs.
\end{lemma}

\begin{proof}
 Let $L$ denote the event that all non-adjacent pairs of vertices have at least $2 \sqrt {(1 + \frac{\epsilon}{4} ) n \log n }$ exclusive fixes in $G_1$, so $L$ holds whp by Lemma~\ref{manyexclusivefixes}. Let $v$, $w_1$,\ldots,$w_4$ be distinct vertices, and let $D_{w_1,\ldots,w_4}^v$ denote the event that $\lbrace v, w_1\rbrace$,\ldots,$\lbrace v, w_4 \rbrace$ are dangerous in $G_2$. Then
\begin{equation}
\label{trallala1}
 D_{w_1,\ldots,w_4}^v \subset L^C \cup (L \cap D_{w_1,\ldots,w_4}^v ).
\end{equation}

Let  $\tilde D_{w_1,\ldots,w_4}^v$ denote the event that  $\lbrace v, w_1\rbrace$,\ldots,$\lbrace v, w_4 \rbrace$ are dangerous in $G_1$. Then, since $D_{w_1,\ldots,w_4}^v \subset \tilde D_{w_1,\ldots,w_4}^v$,
\begin{eqnarray}
\label{trallala2}
\Pb( L \cap D_{w_1,\ldots,w_4}^v) &=& \Pb( \tilde D_{w_1,\ldots,w_4}^v \cap D_{w_1,\ldots,w_4}^v  \cap L) \nonumber\\
&=& \Pb(\tilde D_{w_1,\ldots,w_4}^v \cap L) \Pb (D_{w_1,\ldots,w_4}^v | \tilde D_{w_1,\ldots,w_4}^v \cap L) \nonumber \\
&\leq& \Pb(\tilde D_{w_1,\ldots,w_4}^v) \Pb (D_{w_1,\ldots,w_4}^v | \tilde D_{w_1,\ldots,w_4}^v \cap L). 
\end{eqnarray}
We first want to bound $\Pb(\tilde D_{w_1,\ldots,w_4}^v)$. If $z \in V \setminus \lbrace v, w_1,\ldots,w_4\rbrace$, let $E_z$ be the event that there is a rainbow path of length $2$ from $v$ to at least one $w_i$ via $z$. The edge $vz$ is present in $G_1$ with probability $p_1$, and if it is present, each edge $zw_i$ is present in $G_1$ and has a different colour than $vz$ with probability $\frac{p_1}{2}$, independently. Therefore, $q:=\Pb(E_z)= p_1(1-(1-\frac{p_1}{2})^4) \sim 2p_1^2$, and the events $E_z$ are independent for all $z \in V \setminus \lbrace v, w_1,\ldots,w_4\rbrace$. Let $K$ be the number of vertices $z$ such that $E_z$ holds. If $\lbrace v, w_1 \rbrace$,\ldots,$ \lbrace v,w_4 \rbrace$ are all dangerous pairs, then $K \leq 4d$.

Since $K$ is distributed binomially with parameters $n-5$ and $q$, with mean $(n-5)q \sim 2np_1^2 = 2(1+\epsilon) \log n$, the probability that $K \leq 4d$ is $O(n^{-2(1+\frac{\epsilon}{2})} (\log n)^{4d})$ by Corollary~\ref{binomialbound}. Hence,
\begin{equation}
\label{trallala3}
 \Pb(\tilde D_{w_1,\ldots,w_4}^v) =  O(n^{-2}).
\end{equation}
Conditional on $G_1$, if $\tilde D_{w_1,\ldots,w_4}^v$ and $L$ hold, the probability of the event $N_{\{v,w_i\}} \le d$ that $\lbrace v, w_i \rbrace$ does not get at least $d+1$ of its exclusive fixes in the second round is $O(n^{-0.8})$ by (\ref{atmostdin2ndstep}), and these events are independent for different $w_i$. Therefore, by (\ref{trallala2}) and (\ref{trallala3}),
\[
 \Pb( L \cap D_{w_1,\ldots,w_4}^v) = O(n^{-2} n^{-3.2}) = o(n^{-5}).
\]
Hence, by (\ref{trallala1}),
\begin{eqnarray*}
 \Pb \left( \bigcup_{v, w_1,\ldots,w_4} D_{w_1,\ldots,w_4}^v \right) &\leq& \Pb(L^C) + \Pb \left( \bigcup_{v, w_1,\ldots,w_4} L \cap D_{w_1,\ldots,w_4}^v \right)\\
&\leq& o(1) + n^5 o(n^{-5}) = o(1). 
\end{eqnarray*}
\end{proof}

\begin{lemma}
\label{atmost15}
In $G_2$ whp no vertex is joined by edges to both vertices of more than $3$ vertex disjoint dangerous pairs.
\end{lemma}
\begin{proof}
Let $v$, $u_i$, $w_i$, $i=1,\ldots,4$, be distinct vertices. Let $A$ denote the event that $v$ is adjacent in $G_2$ to all vertices $u_i$ and $w_i$, $i=1,\ldots,4$. Let $D$ denote the event that all pairs $\{u_i,w_i\}$, $i=1,\ldots,4$, are dangerous in $G_2$. Then we want to bound the probability of the event $A \cap D$.

For this, we will explore the edges of $G_2$ in several steps. First reveal the edges of the graph $G_1' = G_1 \setminus \{ v \}$ and their colours. Denote the event that all pairs $\{u_i,w_i\}$, $i=1,\ldots,4$, are dangerous in $G'_1$ by $D'$. Then $D \subset D'$. By a variant of Lemma~\ref{numberofdangerous}, a given pair $\lbrace u_i, w_i \rbrace$ is dangerous in $G_1'$ with probability $O(n^{-\frac{1}{2}(1+\frac{\epsilon}{2})})$, and it is easy to see that $\Pb \left(D' \right)=O(n^{-2(1+\frac{\epsilon}{4})})$. Indeed, for each $z \notin \{v, u_1, w_1,\ldots,u_4, w_4\}$, the probability that $z$ is the middle vertex of a rainbow path joining one of the pairs $\{u_i, w_i\}$, $i=1,\ldots,4$, in $G_1'$ is $2p_1^2(1+o(1))$. These events are independent for different $z$, and at most $4d$ of these events can hold for $D'$ to hold. Since $(n-9) 2p_1^2 \sim 2(1+\epsilon) 
\log n$, by Corollary~\ref{binomialbound}, we have $\Pb(D') = O(n^{-2(1+\frac{\epsilon}{2})} (\log n)^{4d}) = O(n^{-2(1+\frac{\epsilon}{4})})$.

Next, reveal the edges of $G_1$ incident with $v$ and their colours. They are independent from $G'_1$. For $k\in\{0,\ldots,8\}$, let $A_k$ denote the event that $v$ is adjacent in $G_1$ to exactly $k$ of the vertices $\{u_1, w_1,\ldots,u_4, w_4 \}$. Then, since $A_k$ and $D'$ are independent,
\begin{equation}
\label{eqn1}
 \Pb \left( A_k \cap D'\right) \leq  {\binom{8}{k}} p_1^k O( n^{-2(1+\frac{\epsilon}{4})}) =  O( n^{-2-\frac{k}{2}}  ).
\end{equation}
%

As before, let $L$ denote the event that in $G_1$ all non-adjacent pairs of vertices have at least $2 \sqrt {(1 + \frac{\epsilon}{4} ) n \log n }$ exclusive fixes, which holds whp by Lemma~\ref{manyexclusivefixes}. For every pair $\{u_i, w_i\}$, at most two exclusive fixes contain the vertex $v$ (namely $\{v, u_i\}$ and $\{v, w_i\}$). So if $L$ holds, then for $n$ large enough, all pairs $\{u_i, w_i\}$, $i=1,\ldots4$, are either adjacent or have at least $2 \sqrt {(1 + \frac{\epsilon}{8} ) n \log n }$ exclusive fixes which do not contain the vertex $v$. Call these fixes \emph{$v$-free exclusive fixes}.

Now add the edges of $G_2$ not incident with $v$. Let $D''$ denote the event that every pair $\{u_i, w_i\}$, $i=1,\ldots,4$, not adjacent in $G_1$ now gets at most $d$ of its $v$-free exclusive fixes. Note that $D \subset D''$. Conditional on $G_1$, if $L$ holds and $n$ is large enough, every non-adjacent pair $\{u_i, w_i\}$ has at least $2 \sqrt {(1 + \frac{\epsilon}{8} ) n \log n }$ $v$-free exclusive fixes, and each one is added with probability $p_2 \geq 0.4 \sqrt{\frac{\log n}{n}}$, independently. Hence, by Corollary~\ref{binomialbound}, if $L$ and $D'$ hold, 
\[
 \Pb(D'' \mid G_1) = \left(O(n^{-0.8 \sqrt{1 + \frac{\epsilon}{8}}} (\log n)^{d})\right)^4=O(n^{-3.2}).
\]
Finally, we add the remaining edges incident with $v$ in $G_2$. Note that $D''$ depends on ($G_1$ and) the edges of $G_2$ not incident with $v$. Therefore, conditional on $G_1$, $D''$ and $A$ are independent, so if $k \in \{0,\ldots,8\}$, whenever $L$, $D'$ and $A_k$ hold in $G_1$, we have 
\[\Pb(A \cap D\mid G_1)\le \Pb(A\cap D''\mid G_1)=\Pb(A\mid G_1)
\Pb(D''\mid G_1)= p_2^{8-k} O(n^{-3.2}). \]
This gives for $k \in \{0,\ldots,8\}$,
\begin{equation}
\label{eqn3}
 \Pb(A \cap D  \mid A_k \cap L  \cap D')  =  O(n^{-3.2-\frac{8-k}{2}} (\log n)^{\frac{8-k}{2}}).
\end{equation}
Since $A \cap D \subset \left( \bigcup_{k=0}^8 A_k \right) \cap D'$, we have with (\ref{eqn1}) and (\ref{eqn3}),
\begin{eqnarray}
\label{equationon-9}
 \Pb(A \cap D \cap L) &=& \sum_{k=0}^8 \Pb(A \cap D \cap L  \cap A_k \cap D') \nonumber \\
&=&  \sum_{k=0}^8 \Pb( L  \cap A_k \cap D') \Pb( A \cap D  \mid L  \cap A_k \cap D') \nonumber \\
&\leq& \sum_{k=0}^8 \Pb(A_k \cap D') \Pb( A \cap D  \mid L  \cap A_k \cap D') \nonumber \\
&\leq& \sum_{k=0}^8 O( n^{-2-\frac{k}{2}} ) O(n^{-3.2-\frac{8-k}{2}} (\log n)^{\frac{8-k}{2}}) \nonumber \\
&=& O(n^{-9.2} (\log n)^4) = o(n^{-9}).
\end{eqnarray}

Now, since we want to bound the probability that \emph{there exist} vertices  $v$, $u_1$, $w_1$, \ldots, $u_4$, $w_4$ such that $A \cap D$ holds for them, we now add indices $A^{v, (u_i, w_i)_i}$, $D^{(u_i, w_i)_i}$ to our events $A$ and $D$ to make clear which vertices they refer to. The event $L$ is a global event which is the same for all specific vertices $v$, $u_1$, $w_1$,\ldots,$u_4$, $w_4$, so it does not require an index. Then using (\ref{equationon-9}), the probability that there are vertices  $v$, $u_1$, $w_1$,\ldots,$u_4$, $w_4$ such that $A^{v, (u_i, w_i)_i}\cap D^{(u_i, w_i)_i}$ holds is at most
\begin{eqnarray*}
 \lefteqn{\Pb\left(\bigcup_{v, (u_i, w_i)_i}  (A^{v, (u_i, w_i)_i}\cap D^{(u_i, w_i)_i})\right)} \\
&\leq& \Pb(L^C) + \Pb\left(\bigcup_{v, (u_i, w_i)_i}  (A^{v, (u_i, w_i)_i} \cap D^{(u_i, w_i)_i} \cap L)\right) =o(1) + n^9 o(n^{-9}) \\
&=&  o(1),
\end{eqnarray*}
as required.
\end{proof}

\begin{corollary}
\label{atmost15corollary}
In $G_2$ whp no vertex is joined by edges to both vertices of more than $15$ dangerous pairs.
\end{corollary}

\begin{proof}
By Lemma~\ref{atmost15}, whp no vertex is adjacent to both vertices of more than $3$ vertex disjoint dangerous pairs, and by Lemma~\ref{atmost3}, whp every vertex is in at most $3$ dangerous pairs. Assume this from now on.

Note that if a graph has maximum degree at most $\Delta \geq 1$ and more than $t (2 \Delta -1)$ edges, where $t \in \N_0$, then it contains at least $t+1$ pairwise vertex-disjoint edges. This can be seen by induction on $t$ --- note that if one edge and its endpoints are removed from the graph, there are more than $t (2 \Delta -1) - (2 \Delta -1) = (t-1)  (2 \Delta -1)$ edges left.

Therefore, if some vertex $v$ is joined to both vertices of more than $15= 3 \cdot (2 \cdot 3 - 1)$ pairs, and every vertex is in at most $3$ dangerous pairs, then $v$ is joined to both vertices of at least $4 = 3+1$ pairwise disjoint dangerous pairs, which is not possible.
\end{proof}

Recall that we call a non-adjacent pair of vertices \emph{sparsely connected} if they are joined by at most $d=66$ paths of length $2$ (rainbow or otherwise).

\begin{lemma}
\label{atmost1}
Whp every vertex in $G_2$ is in at most one sparsely connected pair.
\end{lemma}

\begin{proof}
Consider some vertex $v$ in $G_2$. Explore $G_2$ in the following way. Explore all edges incident with $v$. By Corollary~\ref{chernoffcorollary}, with probability $1-o(n^{-1})$, we have $|\Gamma(v)| \geq \sqrt{1.98 n \log n}$. Now for every vertex $w \notin \Gamma(v)$ (by definition sparse pairs are not adjacent), the probability that $w$ has at most $d$ edges to $\Gamma(v)$ is $O(e^{-\sqrt{1.98 \cdot 1.99} \log n} (\log n)^{d})= O(n^{-1.98})$ by Corollary~\ref{binomialbound}, and this is independent for different $w$. Therefore, the probability that $v$ is in two sparsely connected pairs is $O(n^2 (n^{-1.98})^2) = O(n^{-1.96})=o(n^{-1})$. Using the union bound, it follows that whp there is no such $v$.
\end{proof}
By Lemma~\ref{atmost3}, Corollary~\ref{atmost15corollary} and Lemma~\ref{atmost1}, the graph $G_2$ with the given edge colouring has property $\mathcal{M}$ whp (with $G_2$ itself as the spanning subgraph), which completes the proof of Proposition~\ref{propertymbutnotd}.

\subsection{Proof of Proposition~\ref{manddimpliesr}}
To prove that $\mathcal{D}$ and $\mathcal{M}$ imply $\mathcal{R}$, we will take the edge 2-colouring given by property $\mathcal{M}$ and re-colour some edges to make a rainbow colouring. We will do this by first re-colouring paths joining sparsely connected dangerous pairs (this step only works if there are such paths at all, i.e., if we have diameter $2$), and then doing the same for richly connected dangerous pairs.

So suppose properties $\mathcal{M}$ and $\mathcal{D}$ hold in some graph $G=(V,E)$. Take the spanning subgraph $G'=(V,E')$ and the edge 2-colouring of $G'$ given by property $\mathcal{M}$. Do not assign colours to the edges in $E \setminus E'$ yet.

We will now assign some colours and change the colours of some edges in $E$ in order to make all dangerous pairs rainbow connected. We will \emph{flag} all edges we (re-)assign a colour to as we go along so that they do not get reassigned another colour later on.

Call a pair of vertices \emph{sparsely sub-connected} if it is sparsely connected in the subgraph $G'$, and call it \emph{richly sub-connected} otherwise. Call a pair \emph{sub-dangerous} if it is dangerous in $G'$. Every sparsely connected pair in $G$ is also sparsely sub-connected. Every dangerous pair in $G$ is also sub-dangerous.

We will start with the sparsely sub-connected sub-dangerous pairs. Take some arbitrary order of these pairs. 

We will go through the sparsely sub-connected sub-dangerous pairs one by one in the given order, and each time ensure there is a rainbow path in $E$ joining them, which is then flagged. Let $\{v,w\}$ be a pair we consider. Since $\mathcal{D}$ holds, either $vw \in E$, in which case we do not need to do anything, or $v$ and $w$ are joined by at least one path of length $2$ in $E$.  Let $vzw$ be such a path.

It is not possible that both of the edges $vz$ and $zw$ are flagged already by the time we look at $\{v,w\}$: Suppose that the edge $e=vz$ is flagged already. This can only have happened in one of the following two ways as shown in Figure~\ref{twocases}. Either there is a vertex $w' \neq w$ such that $\{ v, w'\}$ is sparsely sub-connected and sub-dangerous and the path $vzw'$ was flagged for it, or there is a vertex $z'$ such that $\{z,z'\}$ is sparsely sub-connected and sub-dangerous and the path $zvz'$ was flagged for it. But the first case is impossible because by property $\mathcal{M}$, the vertex $v$ is in at most one sparsely sub-connected pair (namely $\{v,w\}$). So the edge $vz$ was flagged for a sparsely sub-connected sub-dangerous pair $\{z,z'\}$. Similarly, if $zw$ is flagged already, this can only be because there is a vertex $z''$ such that $\{z,z''\}$ is sparsely sub-connected and sub-dangerous and $zw$ was flagged for it. But then $z' \neq z''$, so $z$ is in two sparsely sub-connected pairs, contradicting part (iii) of the definition of $\mathcal{M}$.

\begin{figure}[htb]
\begin{center}
\begin{picture}(300,60)
\unitlength = 1pt
\put(0,0){\circle*{5}}
\put(50,0){\circle*{5}}
\put(100,0){\circle*{5}}
\put(0,0){\line(1,0){100}}
\put(50,0){\line(3,4){30}}
\put(80,40){\circle*{5}}
\put(200,0){\circle*{5}}
\put(250,0){\circle*{5}}
\put(300,0){\circle*{5}}
\put(200,0){\line(1,0){100}}
\put(200,0){\line(-3,4){30}}
\put(170,40){\circle*{5}}
\dottedline{3}(0,0)(80,40) 
\dottedline{3}(170,40)(250,0)
\put(-3,-10){$v$}
\put(96,-10){$w$}
\put(47,-10){$z$}
\put(25,2){$e$}
\put(81,46){$w'$}
\put(197,-10){$v$}
\put(296,-10){$w$}
\put(247,-10){$z$}
\put(221,2){$e$}
\put(163,46){$z'$}
\end{picture}
\end{center}
\caption{\label{twocases}Possible ways in which the edge $e$ could have been flagged before considering $\{v,w\}$. The dotted lines show sub-dangerous pairs other than $\{v,w\}$.}
\end{figure}
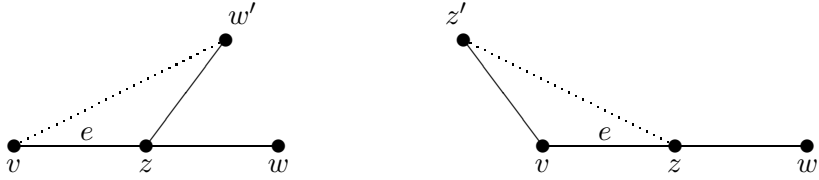

So take the path $vzw$. If necessary, adjust the colour of an un-flagged edge on it to make it a rainbow path, then flag both edges (if they are not flagged already).

Repeat this procedure until all sparsely sub-connected sub-dangerous pairs have rainbow paths. Now we will deal with richly sub-connected sub-dangerous pairs. Again take some arbitrary order of these pairs and consider them one by one.

Let $\{v,w\}$ be the richly sub-connected sub-dangerous pair we consider. By definition, it is either adjacent in $G$, in which case we do not need to do anything, or joined by at least $67$ paths of length $2$ within $E'$. Let $vzw$ be such a path. Then as before and as shown in Figure~\ref{twocases}, the edge $e=vz$ can only be previously flagged for another (sparsely or richly sub-connected) sub-dangerous pair in one of the following two ways. Either there is a vertex $w' \neq w$ such that $\{ v, w'\}$ is sub-dangerous and the path $vzw'$ was flagged for it --- since by property $\mathcal{M}$, $v$ is in at most $3$ sub-dangerous pairs in $G'$, at most $3$ edges in $E'$ incident with $v$ can be flagged this way (now or ever). Or there is a vertex $z'$ such that $\{z,z'\}$ is sub-dangerous and the path $zvz'$ was flagged for it --- since by property $\mathcal{M}$, $v$ is joined by edges to both vertices of at most $15$ dangerous pairs in $G'$, and for each such pair at most $2$ edges incident with $v$ are flagged, at most $30$ edges in $E'$ incident with $v$ can be flagged this way (now or ever).

So at most $33$ edges in $E'$ incident with $v$ can be flagged in this process. Analogously, at most $33$ edges incident with $w$ can be flagged. Since $\{ v,w\}$ is joined by at least $67$ paths of length $2$ in $G'$, there is at least one completely unflagged path at the time we look at $\{v,w\}$. Select one such path for $\{v,w\}$, adjust its colours if necessary to make it a rainbow path, then flag both of its edges and move on to the next richly sub-connected sub-dangerous pair.

Repeat this procedure until all richly sub-connected sub-dangerous pairs have rainbow paths. If there are any uncoloured edges left, assign them arbitrary colours. All sub-dangerous pairs are now joined by rainbow paths. It only remains to check that no non-sub-dangerous pairs have been rainbow disconnected in the process. By the same argument as above (in the description of the procedure for richly sub-connected sub-dangerous pairs), for every vertex $v$ at most $33$ edges incident with $v$ can be flagged and potentially re-coloured by the time we are done. If a pair $\{v,w\}$ is not sub-dangerous, it is either adjacent or is joined by at least $67$ rainbow paths, of which at most $66$ have been re-coloured. Therefore, every previously non-sub-dangerous pair still has at least one rainbow path left, so all pairs of vertices are joined by at least one rainbow path now.
\qed

\end{document}